\documentclass{article}
\usepackage{amsmath,color,amsfonts,amsthm}
\usepackage{amssymb,enumerate,verbatim}
\usepackage{hyperref}

\usepackage{a4wide}
\usepackage{graphicx}
\newtheorem{lemma}{Lemma}
\newtheorem{corollary}[lemma]{Corollary}
\newtheorem{claim}[lemma]{Claim}
\newtheorem{theorem}[lemma]{Theorem}

\newtheorem{conjecture}[lemma]{Conjecture}

\theoremstyle{definition}

\newcommand{\Remark}[1]{}

\title{A counterexample to Stein's Equi-$n$-square Conjecture}
\author{Alexey Pokrovskiy\thanks{Department of Mathematics, ETH, 8092 Zurich, 
Switzerland. {\tt dr.alexey.pokrovskiy@gmail.com}. 
Research supported in part by SNSF grant 200021-175573.}
\and
Benny Sudakov
\thanks{Department of Mathematics, ETH, 8092 Zurich, Switzerland. 
{\tt benjamin.sudakov@math.ethz.ch}. 
Research supported in part by SNSF grant 200021-175573.}}

\date{}

\begin{document}

\maketitle

\begin{abstract}
In 1975 Stein conjectured that in every $n\times n$ array filled with the numbers $1, \dots, n$  with every number occuring exactly $n$ times, there is a partial transversal of size $n-1$. In this note we show that this conjecture is false by constructing such arrays without partial transverals of size $n-\frac{1}{42}\ln n$.
\end{abstract}

\section*{Introduction}
Latin squares are combinatorial objects introduced by Euler in the 18th century. An order $n$ \emph{Latin square} is an $n\times n$ array filled with $n$ symbols such that no symbol appears twice in the same row or column. A \emph{partial transversal} is a collection of cells of the Latin square which do not share the same row, column or symbol. 
Starting with Euler (see \cite{Euler, keedwell2015latin}), questions about transversals in Latin squares were extensively studied.  
The most natural question about them is ``how large a partial transversal one can guarantee to find in every $n\times n$ Latin square?'' A well known conjecture of Ryser, Brualdi, and Stein \cite{Brualdi, Ryser,  Stein}  is that the answer should be $n-1$.

Notice that $n\times n$ Latin squares have the property that every symbol occurs precisely $n$ times. Over 40 years ago, Stein conjectured that this condition on its own is sufficient to guarantee a partial transversal of size $n-1$. In~\cite{Stein}, an \emph{equi-$n$-square} is defined to be an $n\times n$ array filled with $n$ symbols such that every symbol occurs precisely $n$ times, and it is conjectured that every  equi-$n$-square has a partial transversal of size $n-1$. 
\begin{conjecture}[Stein, \cite{Stein}]\label{Conjecture_Stein}
Let $S$ be an $n\times n$ array filled with the symbols $1, \dots, n$ so that each number occurs exactly $n$ times. Then $S$ has a partial transversal of size $n-1$.
\end{conjecture}
Attempts to prove this conjecture have lead to the development of important tools in extremal combinatorics. Stein's Conjecture was the setting of the first application of the Lopsided Lovasz Local Lemma \cite{ErdosSpencer}---Erd\H{o}s and Spencer introduced this variant of the local lemma to show that every $n\times n$ array with $\leq (n-1)/16$ occurences of every symbol has a size $n$ transversal. 
Later Alon, Spencer, and Tetali showed that if there are $\leq \epsilon n$ occurences of every symbol and $n$ is a power of $2$, then the square can be completely decomposed into size $n$ transversals.
When working in equi-$n$-squares, the best currently known result is due to Aharoni, Berger, Kotlar, and Ziv \cite{ABKZ}---they used topological methods to show that such arrays always have a partial transversal of size $2n/3$. This improved on an earlier result of Stein~\cite{Stein} who used the probabilistic method to show that a partial transversal of size $(1-e^{-1})n$ exists in every equi-$n$-square

In this note we produce counterexamples to Conjecture~\ref{Conjecture_Stein}. 
\begin{theorem}\label{Theorem_Counterexamples}
For all sufficiently large $n$, there exist $n\times n$ arrays filled with the symbols $1, \dots, n$ so that each symbol occurs exactly $n$ times, which have no partial transversals larger than $n -\frac{1}{42}\ln n$.
\end{theorem}
We remark that a corollary of the above theorem is that  Erd\H{o}s and Spencer's result cannot be strengthened to hold when the number of occurrences of each symbol in the array is ``$\leq n-\frac{1}{85}\ln n$''  rather than ``$\leq (n-1)/16$''. To see this consider the $n\times n$ array  $S$ from Theorem~\ref{Theorem_Counterexamples} and append $\left \lfloor\frac{1}{84}\ln n \right\rfloor$ rows and columns  to obtain a new array $S'$. Fill the newly created entries with (arbitrarily many) previously unused symbols using every symbol $\leq n$ times. Now $S'$ is an $n'\times n'$ array for $n'=n+\left \lfloor\frac{1}{84}\ln n \right\rfloor$ with $\leq n\leq n'-\frac{1}{85}\ln n'$ occurrences of each symbol (using that $n$ is sufficiently large). It cannot have a size $n'$ transversal, since such a transversal would intersect $S$ in at least $n'-2(n'-n)\geq n -\frac{1}{42}\ln n$ entries, and $S$ was chosen to have no transversal of size $n -\frac{1}{42}\ln n$.

The above theorem is proved in the next section.
These counterexamples still leave open the possibility of Stein's Conjecture holding in some asymptotic sense.
In Section~\ref{SectionRemarks} we discuss some possible asymptotic versions of it.

\section*{Proof}
Our proof relies on the fact that the sequence $a_t=\frac{1}{\sqrt{t}}$ has the property that $b_t=a_t\sum_{i=1}^{t} a_i$ converges as $t\to \infty$ while $c_t=\sum_{i=1}^{t} a_i^2$ diverges. The following lemma proves this in a way which will be convenient to apply. 
\begin{lemma}\label{Lemma_Sequence}
For an integer $n\geq 10^{60}$, consider the sequence $x_t=\left\lfloor \frac{1}{3}\sqrt{\frac{n}{t}}\right \rfloor$. The following hold for all $t$.
\begin{align}
   x_t \sum_{i=1}^t x_i&\leq \frac {n}4 \label{P1}\\
\sum_{i=1}^{n} x_i^2&\geq  \frac{n\ln n}{10} \label{Squares}
\end{align}
\end{lemma}
\begin{proof}
We'll use the fact that for the decreasing function $f(x)$ we have $\int_a^{b} f(x) dx\leq \sum_{i=a}^b f(i)\leq \int_{a-1}^b f(x) dx$.
This implies 
\begin{align}
\frac {1}{3 } \sum_{i=1}^t\sqrt{\frac{n}{i}}
=\frac {\sqrt n}{3 }\left( 1 +   \sum_{i=2}^t\frac{1}{\sqrt i}\right)  
&\leq \frac {\sqrt n}{3 }\left(1 +  \int_{1}^t\frac{1}{\sqrt {x}} dx \right) 
=\frac {\sqrt n}{3 } \left(1 +  2\left(\sqrt{t}-1 \right)\right)
\leq \frac {2\sqrt {n t}}3  \label{SumEstimate}
\end{align}
Now (\ref{P1}) comes from using (\ref{SumEstimate}) and $x_t\leq \frac{1}{3}\sqrt{\frac{n}{t}}$ to get 
$$x_t \sum_{i=1}^t x_i\leq \frac{1}{3}\sqrt{\frac{n}{t}} \left(\sum_{i=1}^t \frac {1}{3 }\sqrt{\frac{n}{i}} \right) \leq \frac{1}{3}\sqrt{\frac{n}{t}} \cdot\frac {2\sqrt{nt}}3\leq \frac n4$$
 For (\ref{Squares}) we have
$$
 \sum_{i=1}^{n} x_i^2
 =\sum_{i=1}^{n} \left\lfloor \frac{1}{3}\sqrt{\frac{n}{i}}\right \rfloor^2
\geq  \frac{n}{9}\sum_{i=1}^{n} \frac{1}{i} -  \frac23\sum_{i=1}^{n} \sqrt{\frac{n}{i}}
 \geq  \frac{n}{9} \int_{1}^{n}\frac{1}{x} dx -  \frac {4n}3 
 = \frac{n\ln n}{9} -  \frac {4n}3\geq  \frac{n\ln n}{10}
$$
The first inequality comes from $\left\lfloor \frac{1}{3}\sqrt{\frac{n}{i}}\right \rfloor\geq \frac{1}{3}\sqrt{\frac{n}{i}}-1$.
The second inequality uses (\ref{SumEstimate}) and the fact that decreasing functions have $\int_a^{b} f(x) dx\leq \sum_{i=a}^b f(i)$.
The last inequality uses $1\leq \frac{\ln n}{120}$ which holds for $n\geq 10^{60}$.
\end{proof}

\begin{figure}
  \centering
    \includegraphics[width=1\textwidth]{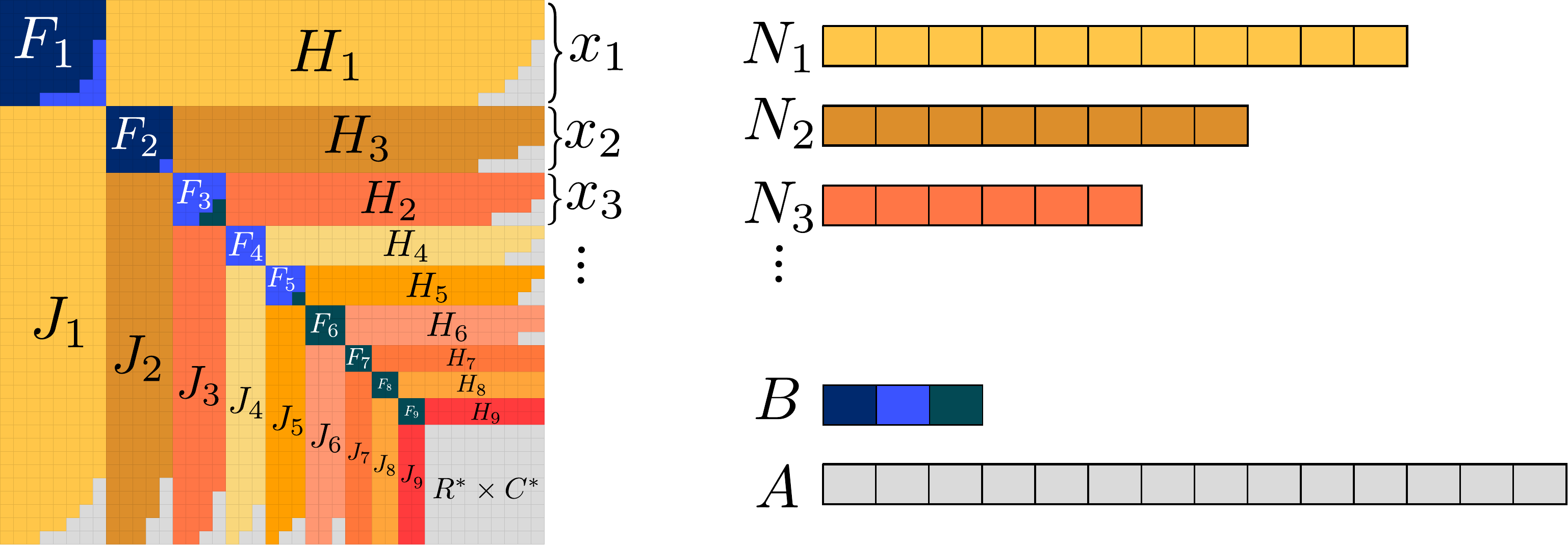}
  \caption{An illustration of a counterexample to Conjecture~\ref{Conjecture_Stein}. The colours in the array represent symbols in the array $S$. For clarity we use the same colour for symbols in each of $N_1, \dots, N_n,$ and $A$. The picture is not entirely to scale since in the actual array in Theorem~\ref{Theorem_Counterexamples}, $R^*\times C^*$ takes up a far larger proportion of the square.}
\label{FigureCounterexample}
\end{figure}
Now we construct counterexamples to Conjecture~\ref{Conjecture_Stein}.
\begin{proof}[Proof of Theorem~\ref{Theorem_Counterexamples}.]
Let $n\geq  10^{60}$, and consider an $n\times n$ array with rows $r_1, \dots, r_n$, and columns $c_1,\dots, c_n$. 
For a set of rows $R$ and a set of columns $C$, we denote the rectangle formed by $R$ and $C$ by  $R\times C=\{(r_i, c_j): r_i\in R, c_j\in C\}$. Let $x_i$ be the sequence from Lemma~\ref{Lemma_Sequence}. Let $n_0$ be the largest number for which $x_{n_0}\neq 0$, and notice that $n_0=\left\lfloor\frac n 9\right\rfloor$. From (\ref{P1}) and the integrality of $x_i$, we have $\sum_{i=1}^{n_0}x_i\leq  \frac{n}{4 x_{n_0}}\leq \frac n4$.
Partition $\{r_1, \dots, r_n\}$ into sets $R_1\cup \dots\cup R_{n_0}\cup R^*$  and $C$ into sets $C_1\cup \dots \cup C_{n_0}\cup C^*$  with $|R_i|=|C_i|=x_i=\left\lfloor \frac{1}{3}\sqrt{\frac{n}{i}}\right \rfloor$ and $|R^*|=|C^*|=n-\sum_{i=1}^{n_0}{x_i}\geq  3n/4$. 
Let $F_i=R_i\times C_i$, $H_i= R_i\times( C^*\cup\bigcup_{j> i} C_j)$,  and   $J_i=(R^*\cup \bigcup_{j> i} R_j)\times C_i$.  Notice that $S$ is the disjoint union of the sets  $F_i, H_i, J_i$ for $i=1, \dots, n_0$, and $R^*\times C^*$.

Notice that $|F_i|=|R_i||C_i|=  x_i^2$.  Using (\ref{P1}) we have
\begin{align*}
|H_i|= |J_i|=|R_i|(n-\sum_{j\leq i} |C_j|)
= x_in - x_i \sum_{j\leq i} x_j
\geq n\left(x_i - \frac14 \right).
\end{align*}
In particular this means that $|H_i\cup J_i|> n(2x_i -1)$.

We now fill $S$ with the $n$ symbols $1, \dots, n$ so that each symbol occurs exactly $n$ times. First split  $\{1, \dots, n\}$  into sets $N_1, \dots, N_{n_0}$ with $|N_i|=2x_i-1$, a set $B$ with $|B|=\left \lfloor \frac{1}{20}\ln n \right\rfloor\geq \frac{1}{21}\ln n$, and  a set $A$ with $|A|=n-|B|-|N_1|-\dots-|N_{n_0}|$ (to see that such a partition is possible, notice that from $\sum_{i=1}^{n_0}x_i\leq  \frac n4$ and $n\geq 10^{60}$, we have $|B|+|N_1|+\dots+|N_{n_0}|=  \left \lfloor \frac{1}{20}\ln n \right\rfloor+ \sum_{i=1}^{n_0} (2x_i-1)\leq \frac{1}{20}\ln n +\frac{n}{2}<n$). 
Fill $S$ as follows
\begin{itemize}
\item For each symbol in $N_i$, place it $n$ times into  $H_i\cup J_i$  ($|H_i\cup J_i|> n|N_i|$ ensures that this is possible).
\item For each symbol in $B$, place it $n$ times into $F_1\cup \dots \cup F_{n_0}$. This is possible since (\ref{Squares}) implies $\sum_{i=1}^{n_0} x_i^2=\sum_{i=1}^{n} x_i^2\geq \frac1{10}{n\ln n}\geq  n|B|$.
\item Place the symbols from $A$ arbitrarily into the remaining entries of $S$ (making sure that there are exactly $n$ occurances of each symbol).
\end{itemize}

Suppose, for the sake of contradiction, that we have a partial transversal $T$ of size  $> n -\frac{1}{42}\ln n$. 
\begin{claim}
If $T$ contains $s$ entries in $F_i$, then $T$ contains at most $2x_i - 2s$ entries in $H_i\cup J_i$
\end{claim}
\begin{proof}
Suppose that $(r_{a_1}, c_{b_1}), \dots, (r_{a_s}, c_{b_s})\in T\cap F_i$. Then since $T$ is a transversal, $T$ cannot have any other entries in rows $r_{a_j}$ or columns $c_{b_j}$ for $j=1, \dots, s$.
Recall that $H_i$ and $F_i$ are both contained in the $x_i$ rows $R_i$, which implies $r_{a_1}, \dots, r_{a_s}\in R_i$. This means that $H_i\cap T$ must be contained in the $x_i-s$ rows $R_i\setminus\{r_{a_1}, \dots, r_{a_s}\}$. Since $T$ has at most one entry in each row we have $|T\cap H_i|\leq x_i-s$. By the same argument we have $|T\cap J_i|\leq x_i-s$.
\end{proof}

Since $|B|\geq  \frac{1}{21}\ln n$, $T$ must contain at least $\frac{1}{42}\ln n$ of the symbols of $B$. Letting $z_i=|T\cap F_i|$, we have $\sum_{i=1}^{n_{0}}z_i\geq \frac{1}{42}\ln n$. By the claim, for all $i$, $T$ has at most $2x_i-2z_i$ entries in $H_i\cup J_i$, and so uses at most $2x_i-2z_i$ symbols in $N_i$ (since the symbols in $N_i$ only occur in $H_i\cup J_i$). Since $|N_i|=2x_i-1$, this means that $T$ doesn't have any entries of at least $|N_i|-(2x_i-2z_i)=2z_i-1$ symbols of $N_i$. Summing up, we have that $T$ doesn't use at least $\sum_{i=1}^{n_{0}}\min(2z_i-1,0)\geq \sum_{i=1}^{n_{0}}z_i\geq \frac{1}{42}\ln n$ symbols. This contradicts  $|T|> n -\frac{1}{42}\ln n$.
\end{proof}

\section*{Concluding remarks}\label{SectionRemarks}
The counterexamples constructed in this note still leave the possibility of Stein's Conjecture holding in some asymptotic sense. There are two natural asymptotic weakenings of the conjecture which may still be true.
\begin{itemize}
\item
{In the setting of Stein's Conjecture,  is there  always a size $n-o(n)$ partial transversal?} This would strengthen the results of  Stein~\cite{Stein} and of Aharoni, Berger, Kotlar, and Ziv \cite{ABKZ}.

It is possible to show that  an asymptotic version of Stein's Conjecture holds with a mild additional condition on the square --- that no symbol appears too often in a row or column. To prove this we will use the following version of R\"odl's Nibble (see eg. \cite{AlonSpencer}).
\begin{theorem}[R\"odl]\label{RodlNibble}
Fix $\epsilon>0$, $r\in \mathbb N$, the following holds for sufficiently large $n$ and $d$.
Let $\mathcal{H}$ be an $r$-uniform, $d$-regular, $n$-vertex hypergraph with every pair of vertices $u,v$ having $d(u,v)\leq o(n)$. Then $\mathcal H$ has a  matching with $(1-\epsilon)n/r$ edges.
\end{theorem}

Using the above result we can prove an asymptotic version of Stein's Conjecture when no symbol appears too often in a row or column.
\begin{corollary}\label{CorollaryAsymptotic}
Fix $\epsilon>0$. Let $S$ be an $n\times n$ array filled with the symbols $1, \dots, n$ so that each number occurs exactly $n$ times in the square, and at most $o(n)$ times in every row and column.
Then $S$ has a partial transversal of size $(1-\epsilon)n$. 
\end{corollary}
\begin{proof}
We define a $3$-uniform, $3$-partite hypergraph $\mathcal H$ as follows. The vertex set of $\mathcal H$ is $\{r_1, \dots, r_n, c_1, \dots, c_n, s_1, \dots, s_n\}$. The edges of $\mathcal H$ are exactly triples of the form $\{r_i, c_j, s_k\}$ with the $(i,j)th$ entry of $S$ being $k$. We claim that $\mathcal H$ satisfies the assumptions of Theorem~\ref{RodlNibble}. Notice that $\mathcal{H}$ is a $3$-uniform, $n$-regular, and $3n$-vertex hypergraph. For a pair of vertices $u,v$ we have $d(u,v)\leq 1$ unless one of $u,v$ is in $\{r_1, \dots, r_n, c_1, \dots, c_n\}$ and the other in  $\{s_1, \dots, s_n\}$. Notice that $d(r_i, s_j)$ and $d(c_i, s_j)$ are equal to the number of occurances of symbol $j$ in row $i$ and column $i$ respectively. By assumption, both of these quantities are at most $o(n)$. 

By Theorem~\ref{RodlNibble}, $\mathcal H$ has a matching $M$ with $(1-\epsilon)n$ edges. Let $T$ be the set of entries in $S$ corresponding to the edges of $M$. Notice that $T$ doesn't have more than one entry with any row, column or symbol because $M$ doesn't have more than one edge through any vertex. Thus $T$ is the required partial transversal.
\end{proof}
The above corollary shows that an asymptotic version of Stein's Conjecture holds when no symbol is repeated more than $o(n)$ times in any row or column. It is easy to check that in the arrays constructed in Theorem~\ref{Theorem_Counterexamples}, each symbol is repeated $O(\sqrt{n})$  in every row or column. 
Thus Corollary~\ref{CorollaryAsymptotic} applies to the squares constructed in Theorem~\ref{Theorem_Counterexamples} to show that they have partial transversals of size $n-o(n)$.  It would be interesting to know it this holds without any restriction on symbol repetitions i.e. if Corollary~\ref{CorollaryAsymptotic} holds without the ``each number occurs at most $o(n)$ times in every row and column'' condition.

\item
{What is the largest $\alpha$ so that every $n\times n$ square with $\leq \alpha n$ occurances of every symbol has a size $n$ transversal?} 
Erd\H{o}s and Spencer~\cite{ErdosSpencer} proved that $\alpha=1/16$ suffices, but it is likely that $\alpha$ can be as large as $1-o(1)$. Again, one can ask an easier question by  adding an extra condition forbidding symbol repetitions in rows and columns. Such a result is true and will be proved in \cite{MPS}:
\begin{theorem}[Montgomery, Pokrovskiy, Sudakov, \cite{MPS}]\label{Theorem_MPS}
Let $S$ be an $n\times n$ array filled with the symbols  so that each symbol occurs $\leq (1-o(1))n$ times in the square, and no symbol is repeated in a row or column. Then $S$ has a transversal.
\end{theorem}
In  \cite{MPS} something stronger is actually shown ---  that under the assumptions of the above Theorem the square has $(1-o(1))n$ disjoint transversals.
Theorem~\ref{Theorem_MPS} shows that if we completely forbid symbol repetitions in rows and columns, then the asymptotic version of Stein's Conjecture holds. Again, it would be interesting to know it Theorem~\ref{Theorem_MPS} holds without any restrictions on repetitions in rows and columns.

\item After Stein made his conjecture, many authors have proposed strengthenings and variations of Conjecture~\ref{Conjecture_Stein}. The construction in this note can be used to disprove most of these conjectures as well. Sometimes this is immediate e.g. the conjecture in \cite{Denes} and Conjecture~1.9 in \cite{AABCKLZ} are direct strengthenings Conjecture~\ref{Conjecture_Stein} and so are false by Theorem~\ref{Theorem_Counterexamples}. Sometimes one needs to modify our construction slightly to disprove related conjectures. For example, Hahn suggested the following conjecture.
\begin{conjecture}[Hahn, \cite{HahnThomassen}]\label{Conjecture_Hahn}
In every edge-colouring of $K_n$ with $\leq n/2-1$ edges of each colour, there is a rainbow path of length $n-1$.
\end{conjecture}
Here ``rainbow path'' means a path in the graph all of whose edges have the same colour. The relationship between this and Stein's Conjecture is that to every symmetric $n\times n$ array $S$, one can assign an edge-coloured complete graph $K_n$ by colouring edge $ij$ with the symbol in the $(i,j)$th entry of $S$ (since $S$ is symmetric, this gives a well-defined colouring of $K_n$). It is easy to see that under this correspondence, partial transversals in $K_n$ correspond to rainbow maximum degree $2$ subgraphs in $K_n$. Thus Conjecture~\ref{Conjecture_Hahn} would imply that every symmetric $n\times n$ array $S$ with $\leq n-1$ occurrences of each symbol has a partial transversal of size $n-1$. The proof of Theorem~\ref{Theorem_Counterexamples} can easily be run so it gives a  symmetric  array i.e. we get graphs satisfying the assumptions of  Conjecture~\ref{Conjecture_Hahn} without rainbow paths longer than $n- \frac1{42}\ln n$.
Here is another conjecture about rainbow subgraphs.
\begin{conjecture}[Aharoni, Barat, Wanless, \cite{ABW}]\label{Conjecture_noncomplete}
Let $G$ be a bipartite  graph with $> \Delta(G) + 1$ edges of each colour. Then $G$ has a rainbow matching using every colour.
\end{conjecture}
To see the relationship between this and Stein's Conjecture: from an $n\times n$ array $S$, build a coloured $K_{n,n}$ with vertices $\{x_1, \dots, x_n, y_1, \dots, y_n \}$ by colouring the edge $x_iy_j\in K_{n,n}$ by the symbol in the $(i,j)$th entry of $S$. It is easy to see that under this correspondance a transversal in $S$ corresponds to a rainbow matching in $K_{n,n}$. 
Consider the  $n\times n$ array $S$ from Theorem~\ref{Theorem_Counterexamples} and the corresponding coloured $K_{n,n}$. Since $S$ has $n$ copies of each symbol, the corresponding $K_{n,n}$ has $n$ edges of each colour. Since $\Delta(K_{n,n})=n$, this is just short of the assumption of Conjecture~\ref{Conjecture_noncomplete} (and so of disproving the conjecture). 

However, it is easy to modify the construction slightly to actually get a counterexample e.g. by deleting the edges of the form $x_iy_i$ and $x_{i}y_{i+1\pmod n}$ in $K_{n,n}$. We use the proof of Theorem~\ref{Theorem_Counterexamples} to get a colouring of this  $(n-2)$-regular graph with $n-2$ colours such that each colour has $n$ edges, but there is no rainbow matching of size  $n- \frac1{42}\ln n$.  This corresponds to deleting two diagonals in the square $S$ in Theorem~\ref{Theorem_Counterexamples}, and checking that the proof still works if we omit these entries: The only parts that need to be checked are that the sets $F_i$ and $H_i\cup J_i$ have enough room to fit the colours they must contain. Specifically we need to check that $|H_i\cup J_i| \geq n(2x_i-1)$ and $|F_1\cup \dots\cup F_{n_0}|\geq n|B|$ still hold after deleting the two diagonals from these sets. This is indeed true because there is room in the inequalities we used for $|H_i\cup J_i|$ and $|F_1\cup \dots\cup F_{n_0}|$. 

The fact that our constructions can disprove Conjecture~\ref{Conjecture_noncomplete} and Conjecture~1.9 in \cite{AABCKLZ} was first noticed by Alon. He also found some interesting further modifications of our construction, see \cite{AlonPC} for details.
\end{itemize}

\subsection*{Acknowledgemnt}
The authors would like to thank Noga Alon for pointing us to the conjectures in \cite{AABCKLZ, ABW}, and for other remarks greatly improving the presentation of this paper.


\begin{thebibliography}{10}
\bibitem{AABCKLZ}
R. Aharoni, N. Alon, E. Berger, M. Chudnovsky, D. Kotlar, M. Loebl,  and R. Ziv.
\newblock Fair representation by independent sets
\newblock {\em In A Journey Through Discrete Mathematics},  31--58 , Springer, Cham., 2017


\bibitem{ABW}
R. Aharoni, J. Barat and I. Wanless.
\newblock Multipartite hypergraphs achieving equality in Ryser’s
conjecture.
\newblock {\em Graphs and Combinatorics.}, 32:1--15, 2016.



\bibitem{ABKZ}
R. Aharoni, E. Berger, D. Kotlar, and R. Ziv.
\newblock On a conjecture of Stein.
\newblock {\em Abh. Math. Semin. Univ. Hambg.}, 87:203--211, 2017.


\bibitem{AlonPC}
N. Alon.
\newblock Personal communication.
\newblock 2018.

\bibitem{AlonSpencer}
N. Alon and J. Spencer.
\newblock The Probabilistic Method.
\newblock {\em John Wiley \& Sons }, 1990.


\bibitem{AlonSpencerTetali}
N. Alon, J. Spencer, and P. Tetali.
\newblock Covering with latin transversals.
\newblock {\em Disc. Appl. Math.}, 57:1--10, 1995.


\bibitem{Brualdi}
R.~A. Brualdi and H.~J. Ryser.
\newblock {\em Combinatorial matrix theory}.
\newblock Cambridge University Press, 1991.


\bibitem{ErdosSpencer}
P Erd\H{o}s and J Spencer.
\newblock Lopsided Lov\'asz local lemma and Latin transversals.
\newblock {\em Disc. Appl. Math.}, 30:151--154, 1991.


\bibitem{Euler}
L. Euler.
\newblock Recherches sur une nouvelle esp\'ece de quarr\'es magiques.
\newblock {\em Verh. Zeeuwsch. Gennot. Weten. Vliss.}, 9:85--239, 1782.


\bibitem{Denes}
J. D\'enes.
\newblock Research problems. 
\newblock {\em Periodica Mathematica Hungarica,}, 17:245--246, 1986.

\bibitem{HahnThomassen}
G. Hahn and C. Thomassen.
\newblock Path and cycle sub-Ramsey numbers and an edge-colouring conjecture. 
\newblock {\em Discrete Math.}, 62:29--33, 1986.



\bibitem{keedwell2015latin}
A.D. Keedwell and J. D{\'e}nes.
\newblock Latin Squares and their Applications.
\newblock {\em Elsevier Science}, 2015.

\bibitem{MPS}
R. Montgomery, A. Pokrovskiy, and B. Sudakov.
\newblock Decompositions into spanning rainbow structures.
\newblock {\em preprint}, 2018.
  


\bibitem{Ryser}
H.~Ryser.
\newblock Neuere probleme der kombinatorik.
\newblock {\em Vortr\"age \"uber Kombinatorik, Oberwolfach},  69--91,
  1967.

\bibitem{Stein}
S.~K. Stein.
\newblock Transversals of {L}atin squares and their generalizations.
\newblock {\em Pacific J. Math.}, 59:567--575, 1975.




\end{thebibliography}
\end{document}